\documentclass{l4dc2024}

\title[Model-Free \(\mu\)-Synthesis: A Nonsmooth Optimization Perspective]{Model-Free \(\mu\)-Synthesis: A Nonsmooth Optimization Perspective}

\usepackage{times}

\usepackage{float}

\usepackage{bm}

\usepackage{array}
\usepackage{xcolor}
\usepackage{url}            
\usepackage{booktabs}       
\usepackage{amsfonts}       
\usepackage{nicefrac}       
\usepackage{microtype}      
\usepackage{amsmath}

\usepackage{algorithm} 
\usepackage{algorithmic} 
\usepackage{xcolor}

\usepackage{dsfont}
\newcommand{\norm}[1]{\left\|#1\right\|}
\newcommand{\field}[1]{\mathbb{#1}}
\newcommand{\R}{\field{R}}
\usepackage{enumitem}
\DeclareMathOperator*{\dist}{dist}
\newcommand{\bmat}[1]{\begin{bmatrix}#1\end{bmatrix}}
\newcommand{\bmtx}[1]{\left[\begin{array}{#1}}
\newcommand{\emtx}{\end{array}\right]} 
\newcommand{\bsmtx}{\left[ \begin{smallmatrix}}
\newcommand{\esmtx}{\end{smallmatrix} \right]} 

\author{%
 \Name{Darioush Keivan}  \Email{dk12@illinois.edu}\\
 \Name{Xingang Guo} \Email{xingang2@illinois.edu} \\
\Name{Peter Seiler} \Email{pseiler@umich.edu}\\
\Name{Geir Dullerud} \Email{dullerud@illinois.edu}\\
\Name{Bin Hu} \Email{binhu7@illinois.edu}
}

\begin{document}

\maketitle

\begin{abstract}%
In this paper, we revisit model-free policy search on 
an important robust control benchmark, namely  \(\mu\) synthesis. In the general output-feedback setting, there do not exist convex formulations for this problem, and hence global optimality guarantees are not expected. \citet{apkarian2011nonsmooth} presented a nonconvex nonsmooth policy optimization approach for this problem, and achieved state-of-the-art design results via using subgradient-based policy search algorithms which generate update directions in a model-based manner. Despite the lack of convexity and global optimality guarantees, these subgradient-based policy search methods have led to impressive numerical results in practice.
Built upon such a policy optimization persepctive, our paper extends these subgradient-based search methods to a model-free setting. Specifically, we examine the effectiveness of two model-free policy optimization strategies:  the model-free non-derivative sampling method and the zeroth-order policy search with uniform smoothing.  We performed an extensive numerical study to demonstrate that both methods consistently replicate the design outcomes achieved by their model-based counterparts. Additionally, we provide some theoretical justifications showing that convergence guarantees to stationary points can be established for our model-free $\mu$-synthesis under some assumptions related to the coerciveness of the cost function. 
Overall, our results demonstrate that derivative-free policy optimization offers a competitive and viable approach for solving general output-feedback \(\mu\)-synthesis problems in the model-free setting.

\end{abstract}


\section{Introduction}

In recent years, the empirical success of reinforcement learning (RL) has significantly impacted the controls field, sparking increased interest in direct policy search methods. Various properties of policy optimization (PO) have been established across many standard control benchmark problems~\citep{hu2022towards}, including linear quadratic regulators (LQR) \citep{pmlr-v80-fazel18a, bu2019lqr,malik2019derivative, yang2019provably, mohammadi2021convergence, furieri2020learning, hambly2021policy,fatkhullin2020optimizing,duan2021optimization}, stabilization \citep{perdomo2021stabilizing, ozaslan2022computing}, linear robust/risk-sensitive control \citep{zhang2021policy,zhang2020stability,zhang2021derivative,gravell2020learning,zhao2021primal,guo2022global,guo2023complexity,tang2023global}, linear quadratic Gaussian (LQG) \citep{zheng2021analysis,zheng2023benign,zheng2022escaping,hu2022connectivity}, and Markov jump linear quadratic control \citep{jansch2020convergence,jansch2022policy,jansch2020policyMDP,rathod2021global}. Many of the above results (implicitly) rely on the fundamental connections between the nonconvex policy optimization formulations and the existing higher-dimensional convex synthesis reformulations \citep{scherer1997multiobjective, boyd1994linear, Gahinet1994, Scherer2004}. 
However, there are important linear robust control problems that do not have convex reformulations in the first place. 
In this paper, we will look at one of such problems, namely the general output-feedback $\mu$-synthesis. 
The $\mu$-synthesis has long been a cornerstone in robust control, dealing with systems affected by uncertainties \citep{packard93musyn, honda14ACC_HDD,zhou1996robust,dullerud2013course}. The goal of $\mu$-synthesis is to design a controller that stabilizes the closed-loop dynamics and minimizes the so-called structured singular value (or equivalently robust performance) at the same time. 
Traditional methods for addressing $\mu$-synthesis have typically centered around finding upper bounds using \(DK\)-iteration techniques~\citep{zhou1996robust}. Later, \cite{apkarian2011nonsmooth} innovatively reformulated the general $\mu$-synthesis as a nonconvex, nonsmooth, model-based policy optimization problem such that subgradient-based search techniques can be directly applied to achieve state-of-the-art results and even outperform \(DK\)-iteration on many examples.
Dealing with the nonconvex nonsmooth optimization in \cite{apkarian2011nonsmooth} is highly non-trivial, i.e.  one needs to enlarge the Clarke subdifferential in some novel way for the purpose of generating good descent directions. The original work in \cite{apkarian2011nonsmooth} relies on a frequency-domain technique (which is quite similar to the \texttt{Hinfstruct} solver \citep{gahinet2011structured} from the MATLAB robust control toolbox). In principle, one can also use the gradient sampling technique   \citep{burke2020gradient,burke2005robust,kiwiel2007convergence}\footnote{Interestingly,
 the \(\mathcal{H}_\infty\) fixed-order optimization (\texttt{HIFOO}) solver \citep{arzelier2011h2,gumussoy2009multiobjective}) was developed based on such an alternative choice of algorithms.}. 
In this work, we extend the nonsmooth optimization perspective on $\mu$-synthesis \citep{apkarian2011nonsmooth} to the model-free setting. 
Notice that the model-free state-feedback \(\mu\)-synthesis has been previously addressed via combining \(DK\)-iteration and a central-path algorithm that adopts robust adversarial reinforcement learning (RARL) as subroutines for finding analytical center in the \(K\) step~\citep{keivan2021model}. 
However, such an approach is not directly applicable in the general output-feedback setting. 
Alternatively, we examine the effectiveness of two model-free policy search strategies that are deeply connected to nonsmooth optimization theory: the model-free non-derivative sampling method and the zeroth-order policy search with uniform smoothing.

Our findings indicate that both methods consistently yield design solutions comparable to those achieved by conventional model-based approaches, such as \(DK\)-iteration method. Similar to their model-based counterparts \citep{apkarian2011nonsmooth}, our proposed model-free methods even outperform \(DK\)-iterations in some cases. Additionally, we provide some theoretical justifications showing that convergence guarantees to stationary points can be established for our model-free $\mu$-synthesis under some assumptions related to the coerciveness of the cost function. Our theoetical developments extend recently-developed convergence/complexity results on $\mathcal{H}_\infty$ policy search~\citep{guo2022global,guo2023complexity}. These outcomes underscore the potential of direct policy search via zeroth-order optimization as a viable and competitive approach for addressing the general output-feedback $\mu$-synthesis in the model-free~setting.

\section{Problem Formulations and Preliminaries} \label{sec:2}
\paragraph{Setup of $\mu$-synthesis.} For the rest of this paper, let $G$ denote the following linear time-invariant (LTI) system:
\begin{align}
\label{eq_system model}
\begin{split}
    x_{t+1} &= A x_t + B_w w_t + B_d d_t + B_u u_t, \\
v_{t} &= C_v x_t + D_{vw} w_t + D_{vd} d_t + D_{vu} u_t, \\
e_t &= C_e x_t + D_{ew} w_t + D_{ed} d_t + D_{eu} u_t, \\
y_t &= C_y x_t + D_{yw} w_t + D_{yd} d_t + D_{yu} u_t,
\end{split}
\end{align}
where $x_t\in \R^{n_x}$ is the system state, $u_t\in\R^{n_u}$ is the control input, $d_t \in \R^{n_d}$ is the exogenous disturbance, $e_t \in \R^{n_e}$ is the performance signal, $y_t \in \R^{n_y}$ is the
output measurement, \(v_t \in \R^{n_v}\) is the uncertainty input, and \( w_t \in \R^{n_w} \) is the uncertainty output.
To start, we consider the standard robust synthesis interconnection as shown in the left-side sub-figure of Figure~\ref{fig:musyn}.  Let \(\mathcal{F}_l(G,K)\) denote the feedback interconnection of \(G\) and \(K\).
The pair \( (v, w) \) satisfies the relation \( w = \Delta (v) \), where \( \Delta \) is a mapping within a \emph{cone} \( \mathbf{\Delta} \) of structured bounded linear time-invariant (LTI) operators. 
The term \(\mathbf{\Delta}\) denotes the uncertainty set \citep{zhou1996robust,dullerud2013course}. The main objective of robust synthesis is to design a controller that stabilizes the closed-loop dynamics and optimizes the robust performance\footnote{A controller $K$ achieves \emph{Robust Performance of level $\gamma$} if for all $\Delta \in \mathbf{\Delta}$ satisfying $\|\Delta\|_\infty \le \frac{1}{\gamma}$, the
closed-loop system is well-posed, stable, and has the mapping from $d$ to $e$ satisfying
$\|T_{d\mapsto e}(\Delta)\|_\infty \leq \gamma$.}  at the same time.  However, verifying robust performance is inherently intractable, prompting a shift in focus towards establishing an upper bound. This is achieved by introducing a set of positive scaling functions \(\mathbf{D}\), each satisfying \(D\Delta = \Delta D\) for all \(\Delta \in \mathbf{\Delta}\). The block diagram of the scaled system is shown in the middle sub-figure of Figure~\ref{fig:musyn}. To optimize this upper bound on robust performance, one eventually needs to minimize the \(\ell_2\) gain from scaled inputs \((\tilde{w},d)\) to scaled outputs \((\tilde{v},e)\) using appropriate \(D\)-scales. Therefore, the original robust synthesis task reduces to solving the following optimization problem:
\begin{equation} 
\label{eq:opt_up}
    \inf_{K \in \mathcal{K}, D \in \mathbf{D}}  \| D \, \mathcal{F}_l(G,K) \, D^{-1} \|_\infty.
\end{equation}
In classical \(\mu\)-synthesis, the optimization problem \eqref{eq:opt_up} is tackled via the \(DK\)-iteration approach, which alternates between optimizing \(D\) while fixing \(K\) as constant, and then optimizing \(K\) with \(D\) held constant. During the \(K\)-step, \(K\) is determined through an \( \mathcal{H}_{\infty} \) synthesis procedure applied to the scaled plant outlined in \eqref{eq:opt_up}. During the \(D\)-step, a realizable \(D\)-scaling is obtained by optimizing \( D \) over a discrete frequency grid and subsequently fitting a transfer function $D \in \mathcal{RH}_{\infty}$  with $D^{-1} \in \mathcal{RH}_{\infty}$. While this heuristic technique can locate effective solutions in many practical situations, the coordinate descent nature can potentially lead to unnecessary conservatism.

\begin{figure}[t!]
\centering
\begin{minipage}{0.3\textwidth}
\centering
\scalebox{0.8}{
\begin{picture}(172,129)(23,-15)
 \thicklines
 \put(75,25){\framebox(40,40){$G$}}
 \put(163,42){$d$}
 \put(155,45){\vector(-1,0){40}}  
 \put(23,42){$e$}
 \put(75,45){\vector(-1,0){40}}  
 \put(80,75){\framebox(30,30){$\Delta$}}
 \put(42,70){$v$}
 \put(55,55){\line(1,0){20}}  
 \put(55,55){\line(0,1){35}}  
 \put(55,90){\vector(1,0){25}}  
 \put(143,70){$w$}
 \put(135,90){\line(-1,0){25}}  
 \put(135,55){\line(0,1){35}}  
 \put(135,55){\vector(-1,0){20}}  
 \put(80,-15){\framebox(30,30){$K$}}
 \put(42,18){$y$}
 \put(55,35){\line(1,0){20}}  
 \put(55,0){\line(0,1){35}}  
 \put(55,0){\vector(1,0){25}}  
 \put(143,18){$u$}
 \put(135,0){\line(-1,0){25}}  
 \put(135,0){\line(0,1){35}}  
 \put(135,35){\vector(-1,0){20}}  
\end{picture}
}

\end{minipage}
\hfill
\begin{minipage}{0.3\textwidth}
\centering
\scalebox{0.8}{
\begin{picture}(172,129)(23,-7)
 \thicklines
 \put(75,25){\framebox(40,50){$G$}}
 \put(30,54){\framebox(25,20){$D$}}
  \put(135,54){\framebox(25,20){$D^{-1}$}}
 \put(193,42){$d$}
 \put(185,45){\vector(-1,0){70}}  
 \put(-7,42){$e$}
 \put(75,45){\vector(-1,0){70}}  
 \put(75,63){\vector(-1,0){20}} 
  \put(-7,63){$\tilde{v}$}
  \put(193,63){$\tilde{w}$}
   \put(122,69){$w$}
     \put(62,69){$v$}
  \put(30,63){\vector(-1,0){25}} 
 \put(135,63){\vector(-1,0){20}}  
  \put(185,63){\vector(-1,0){25}} 
 \put(80,-15){\framebox(30,30){$K$}}
 \put(42,18){$y$}
 \put(55,35){\line(1,0){20}}  
 \put(55,0){\line(0,1){35}}  
 \put(55,0){\vector(1,0){25}}  
 \put(143,18){$u$}
 \put(135,0){\line(-1,0){25}}  
 \put(135,0){\line(0,1){35}}  
 \put(135,35){\vector(-1,0){20}}  
\end{picture}
}
\label{fig:DK}
\end{minipage}
\hfill
\begin{minipage}{0.3\textwidth}
\centering
\scalebox{0.8}{
\begin{picture}(172,129)(23,-4)
 \thicklines
 \put(78,45){\framebox(45,45){$G_c$}} 
 \put(165,82){$\bmat{\tilde{w}\\d}$}
 \put(160,82){\vector(-1,0){37}}  
 \put(18,82){$\bmat{\tilde{v}\\e}$}
 \put(78,82){\vector(-1,0){37}}  
 \put(78,-12){\framebox(45,45){$K_c$}} 
 \put(35,40){$y_c$}
 \put(53,53){\line(1,0){25}}  
 \put(53,25){\line(0,1){28}}  
 \put(53,25){\vector(1,0){25}}  
 \put(157,40){$u_c$}
 \put(148,25){\line(-1,0){25}}  
 \put(148,25){\line(0,1){28}}  
 \put(148,53){\vector(-1,0){25}}  
\end{picture}

}
\label{fig:D_scale}

\end{minipage}
\caption{Interconnection for Robust Synthesis}
\label{fig:musyn}
\end{figure}
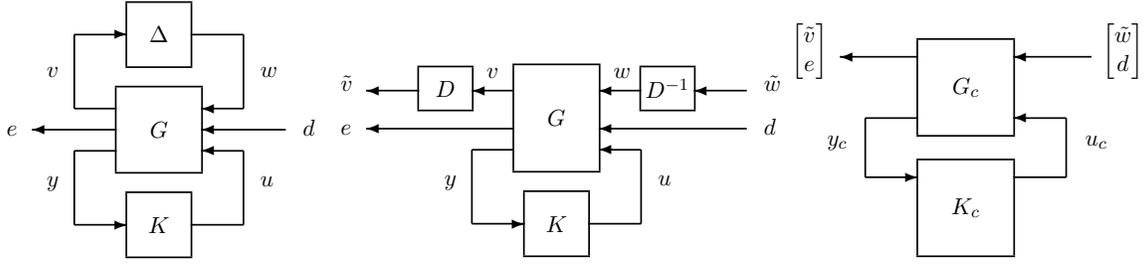

\paragraph{Policy optimization formulation.}
In contrast to the \(DK\)-iteration approach,~\cite{apkarian2011nonsmooth} proposed an algorithm to solve for \(D\) and \(K\) simultaneously by constructing a feedback interconnection of an augmented plant and a structured controller, as shown in the right-side sub-figure of Figure~\ref{fig:musyn}. This innovative approach transforms the robust synthesis problem \eqref{eq:opt_up} into a structured \(\mathcal{H}_{\infty}\) synthesis problem. In this framework, the augmented controller \(K_c\) and the augmented plant  \( G_c \) are formulated~as
\begin{align*}
    G_c = \left[ \begin{array}{cc}
I & 0 \\
\hline
0 & I \\ 0 & 0\\I & 0
\end{array} \right] G \left[ \begin{array}{c|ccc}
I & 0 & -I & 0 \\
0 & I & 0 & 0
\end{array} \right] + \left[ \begin{array}{c|ccc}
0 & 0 & 0 & I \\
\hline
0 & 0 & 0 & 0\\I & 0 & -I & 0 \\0 & 0 & 0 & 0
\end{array} \right],\,     K_c = \begin{bmatrix}
        K & 0 & 0 \\ 0 & \tilde{D} & 0 \\ 0 & 0 & \tilde{D}
    \end{bmatrix},
\end{align*}
where \(\tilde{D}\) is defined as \(\tilde{D} := \bsmtx D & 0 \\ 0 & I \esmtx - I\). Specifically, we can reformulate \eqref{eq:opt_up} as the following policy optimization problem
\begin{equation} 
\label{eq:muopt}
 \min_{K_c \in \mathcal{K}_c}  J(K_c),
\end{equation}
where the decision variable $K_c$ is determined by the structural controller parameterization, the cost function $J(K_c)$ is the closed-loop $\mathcal{H}_\infty$ norm (from $\bsmtx \tilde{w} \\d \esmtx$ to $\bsmtx \tilde{v} \\ e \esmtx$) for a given controller, and the feasible set $\mathcal{K}_c:=\{K_c: \,F_l(G_c,K_c) \,\mbox{is internally stable}\}$ just carries the information of the closed-loop stability constraint. The above PO formulation allows very flexible choices in the controller parameterization. For example, in the static output feedback setting and static \(D\) scaling, the augmented controller is just parameterized by a static matrix $K_c$. In the dynamic output feedback setting and LTI \(D\) scaling, the augmented controller has the following state-space form: 
\begin{align}
\begin{split}
\label{eq:Gmodel}
    \xi_{t+1}&=\bmat{A_K & 0 & 0 \\ 0 & A_D & 0 \\ 0 & 0 & A_D} \xi_t+\bmat{B_K & 0 & 0\\0 & \tilde{B}_D & 0\\ 0 &0 & \tilde{B}_D} {y_c}_t,\\
    {u_c}_t&=\bmat{C_K & 0 & 0 \\ 0 & \tilde{C}_D & 0 \\0 & 0 & \tilde{C}_D} \xi_t + \bmat{D_K & 0 & 0 \\ 0 & \tilde{D}_D & 0 \\ 0 & 0 & \tilde{D}_D}  {y_c}_t,\\
    \end{split}
\end{align}
where $\xi_t$ is the internal state for the augmented controller, \(\tilde{B}_D = \bsmtx B_D & 0 \esmtx\), \(\tilde{C}_D = \bsmtx C_D \\ 0 \esmtx\), \(\tilde{D}_D = \bsmtx D_D-I & 0 \\0 & 0 \esmtx\), and the decision variable $K_c$ is just the tuple \((A_K,B_K,C_K,D_K,A_D,B_D,C_D,D_D)\). In general, the formulation \eqref{eq:muopt} provides a unified paradigm for robust control synthesis \eqref{eq:opt_up} by allowing flexible choices of the augmented controller parameterization. When the system model is known, subgradient information can be efficiently calculated for solving \eqref{eq:muopt}. \citet{apkarian2011nonsmooth} takes such a nonsmooth optimization approach, leading to state-of-the-art numerical results on many examples. 

\paragraph{Problem statement: model-free policy search.} In this paper, we focus on solving \eqref{eq:muopt} under the model-free setting, i.e. the plant $G$ is unknown. 
We do not even assume any prior knowledge on the order of $G$, i.e. the state dimension $n_x$ can be unknown\footnote{It is well-known that the lack of order information can cause difficulty for system identification in some cases.}. We are particularly interested in model-free direct policy search which updates recursively as follows
\begin{align}
K_c^{n+1}=K_c^n-\alpha^n F^n,
\end{align}
where $F^n$ is a descent direction generated from some simulated data of the closed-loop system $F_l(G_c,K_c^n)$. 
If the cost $J$ is smooth around $K_c^n$, then obviously we can set $F^n$ to be some sample-based estimation of the gradient $\nabla J(K_c^n)$. However, the closed-loop $\mathcal{H}_\infty$ norm is typically a nonconvex nonsmooth function of $K_c$, and can be non-differentiable over important feasible points (e.g., stationary points) in the policy space \citep{apkarian2006controller,apkarian2006nonsmooth,arzelier2011h2,gumussoy2009multiobjective,burke2020gradient,curtis2017bfgs}. Advanced subgradient-based optimization techniques are typically needed for solving such nonconvex nonsmooth PO problems. We will study how to compute $F^n$ in the model-free setting.

\paragraph{Review: Subgradient methods in the model-based setting.}
For readability, let us briefly review  several subgradient-based methods that have been used in the model-based setting.
It is known than the closed-loop \(\mathcal{H}_{\infty}\) objective function \eqref{eq:opt_up} is locally Lipschitz\footnote{A function $J:\mathcal{K}_c \rightarrow \R $ is said to be locally Lipschitz if for any bounded set $S \subset \mathcal{K}_c$, there exists a constant $L$ such that $|J(K_c) - J(K'_c)| \leq L \|K_c-K'_c\|_2$ for all $K_c,K'_c \in S$.} over the feasible set \(\mathcal{K}_c\) \citep{apkarian2006nonsmooth}. For a locally Lipschitz function, the Clarke subdifferential exists and is defined~as
\begin{align}
    \partial J(K_c):= \text{conv} \{\lim_{i \rightarrow \infty} \nabla J({K_c}_i): {K_c}_i \rightarrow K_c, {K_c}_i \in \text{dom}(\nabla J)\subset \mathcal{K}_c \},
\end{align}
where $\text{conv}$ stands for the convex hull.  Dealing with the nonconvex nonsmooth optimization in \cite{apkarian2011nonsmooth} is highly non-trivial, i.e.  one needs to enlarge the Clarke subdifferential in some novel way for the purpose of generating good descent directions. The original work in \cite{apkarian2011nonsmooth} relies on a frequency-domain technique (which is quite similar to the \texttt{Hinfstruct} solver \citep{gahinet2011structured} from the MATLAB robust control toolbox). 
Alternatively, gradient sampling methods can also generate good descent directions \citep{burke2020gradient,burke2005robust,kiwiel2007convergence}. 
Specifically, notice that the 
Goldstein $\delta$-subdifferential for a point $K_c\in \mathcal{K}_c$ is defined as 
    \begin{align}\label{goldstein_sub}
        \partial_{\delta}J(K_c):= \text{conv}\{ \cup_{K_c' \in \mathcal{B}_{\delta}(K_c)} \partial J(K'_c) \},
    \end{align}
    where $\mathcal{B}_{\delta}(K_c)$ denotes the $\delta$-ball around $K_c$, and is implicitly required to be in $\mathcal{K}_c$. Clearly, $\partial_{\delta} J(K_c)$ is much larger than $\partial J(K_c)$. The minimum norm element of the Goldstein subdifferential provides a good descent direction, i.e. we have $J(K_c-\delta F/\|F\|_2) \leq J(K_c) - \delta \|F\|_2$ for $F$ being the minimum norm element of $\partial_\delta J(K_c)$ \citep{goldstein1977optimization}. Computing the minimum norm element from the Goldstein subdifferential can be difficult, and the main idea of gradient sampling method is to estimate a good descent direction from approximating $\partial_\delta J(K_c)$ as the convex hull of randomly sampled gradients over $\mathcal{B}_\delta(K_c)$ (this is reasonable due to the fact that a locally Lipschitz function is differentiable almost everywhere). At every iteration step $n$, one can randomly sample differential points around $K_c^n$, and use the convex hull formed by the gradients at those sampled points to approximate the Goldstein subdifferential. Then the minimum norm element from this convex hull of sampled gradients can be efficiently computed via a convex quadratic program and serves as a good descent direction for the policy update.

\begin{algorithm}[t!]
\caption{Non-derivative Sampling (NS)}
\label{alo:NS}
\begin{algorithmic}
\STATE{\bfseries Require:} initial stabilizing policy $K_c^0 \in \mathcal{K}_c$, initial sampling radius $\delta^0$,  optimality tolerances $\delta_{opt}, \epsilon_{opt} > 0$, initial stationarity target $\epsilon^0 \in [0,\infty)$, reduction factors $\mu_{\delta}, \mu_{\epsilon} \in (0,1]$, problem dimension $d$, line search parameters $(\beta,\underline{t},\kappa)$ in  $(0,1)$, and a sequence of positive mollifier parameters defined as $\alpha^n = \alpha_0/(n+1)$. 
\FOR{$n=0,1,2,\cdots$}  
\STATE Independently sample $\{ K_c^{n,i} \}_{i=1}^{d+1}$ uniformly from $\mathcal{B}_{\delta^n}(K_c^n)$
\STATE Independently sample $\{ z^{n,i}\}_{i=1}^{d+1}$ uniformly from $\mathcal{Z}:= \prod_{i=1}^d [-1/2, 1/2]^d.$
\STATE Compute $F^n = \arg \min \frac{1}{2} \|F\|_2^2 \, \text{s.t.}\, F \in \text{conv}\{ \chi(K_c^{n,1},\alpha^n,z^{n,1}), \cdots, \chi(K_c^{n,m},\alpha^n,z^{n,m}) \}.$
\STATE \textbf{if} {$\|F^n\| \le \epsilon_{opt}$ and $\delta^n \le \delta_{opt}$}, terminate.
\IF{$\| F^n \| \le \epsilon^n$}
\STATE set $\epsilon^{n+1}  \leftarrow \mu_{\epsilon} \epsilon^n$, $\delta^{n+1}  \leftarrow \mu_{\delta} \delta^n$, $t^n \leftarrow 0$, $K_c^{n+1}\leftarrow K_c^n$, and move to the next round
\ELSE 
\STATE set $\delta^{n+1} \leftarrow \delta^n$, $\epsilon^{n+1}  \leftarrow \epsilon^n$, $\hat{F}^n \leftarrow F^n / \| F^n\|_2$, and $K_c^{n+1}\leftarrow K_c^n-t^n \hat{F}^n$, where $t^n$ is determined using the following line search strategy:
\STATE \qquad (i) Choose an initial step size $t = t_{ini}^n  = \delta^n\ge t^n_{\min} :=  \min\{\underline{t}, \kappa \delta^n/3\}$ \\
\STATE \qquad (ii) If $J(K_c^n - t \hat{F}^n  ) \le J(K_c^n) - \beta t \| F^n \|$, return $t^n := t$  \\
\STATE \qquad (iii) If $\kappa t < t^n_{\min}$,  return $t^n := 0$ \\

\STATE \qquad (iv) Set $t:= \kappa t$, and go to (ii). 
\ENDIF
\ENDFOR
\end{algorithmic}
\end{algorithm}

\section{Main Results: Model-Free Algorithms and Theoretical Justifications}
In this paper, we aim to minimize the cost function \eqref{eq:muopt} within the policy space directly utilizing two distinct zeroth-order optimization methods: the non-derivative sampling (NS) and the standard zeroth-order policy search with randomized smoothing. In addition, we assume that the state, input, and output matrices specified in \eqref{eq_system model} are unknown, and the cost function \eqref{eq:muopt} can only be inferred from the input/output data via a ``black-box" simulator of the underlying system. In particular, we employ the model-free time-reversal power-iteration-based $\mathcal{H}_{\infty}$ estimation methods proposed in~\cite{Wahlberg2010} for estimating the cost function value \eqref{eq:muopt}. Furthermore, we demonstrate theoretically that under some assumptions, the PO problem \eqref{eq:muopt} can be rewritten as another PO problem with coercive cost function. Consequently, leveraging the techniques described in~\citet{guo2022global} and \citet{guo2023complexity}, we can obtain some theoretical justifications for  convergence to stationary points.

\subsection{Non-derivative Sampling}
As mentioned previously, gradient sampling (GS) is a principal optimization algorithm utilized in the HIFOO robust control package \citep{arzelier2011h2,gumussoy2009multiobjective}. In this work, given that the system model is unknown and we have access only to estimates of the cost function \eqref{eq:opt_up}, we adopt the NS algorithm \citep{kiwiel2010nonderivative}, a derivative-free counterpart to the GS algorithm. In contrast to GS, the NS estimates the gradient from function values via Gupal’s estimation $\chi(K,\alpha, z)$ (See \cite[Section 2]{kiwiel2010nonderivative} for more details on the computation of $\chi$). The NS method can be implemented as outlined in Algorithm \ref{alo:NS}.

In the model-free setting, one has to estimate the cost function from data. As discussed in the beginning of this section, we will  use well-established $\mathcal{H}_\infty$-norm estimation methods such as the power iteration method in~\citet{Wahlberg2010} to estimate the cost values.

\subsection{Derivative-free Optimization with Randomized Smoothing }
Our second derivative-free method is based on the utilization of randomized smoothing techniques, which have been widely adopted in both convex and nonconvex optimization challenges \citep{Duchi2012,Ghadimi2013}. We define the uniformly randomized smoothed counterpart of $J(K_c)$ as below.
\begin{definition} \label{uniform_sample}
Given a function $J$ that is $L$-Lipschitz (which may be nonconvex or nonsmooth) and a uniform distribution $\mathbb{P}$ over the set $\{ U: \| U\|_F = 1 \}$, the uniformly smoothed form of $J$, denoted as $J_\delta$, is given by 
\begin{equation} \label{eq:J_delta}
    J_\delta(K_c) = \mathbb{E}_{U\sim \mathbb{P}}[J(K_c+\delta U)].
\end{equation}
\end{definition}
This definition requires that both \(K_c\) and the perturbed \(K_c + \delta U\) remain within the feasible set \(\mathcal{K}_c\), for every \(U\) drawn from the distribution \(\mathbb{P}\). Recent insights from ~\cite{Lin2022} illustrate a key relationship between the Goldstein subdifferential and uniform smoothing, highlighting that \(\nabla J_\delta(K_c) \) is an element of \(\partial_\delta J(K_c)\). Under the definition of the Goldstein \(\delta\)-subdifferential \eqref{goldstein_sub}, a point \(K_c\) is a \((\delta,\epsilon)\)-stationary point if \(\dist(0, \partial_\delta J(K_c)) \le \epsilon\). Therefore, an $\epsilon$-stationary point of \(J_\delta(K_c)\) is also \((\delta,\epsilon)\)-stationary for the original function \(J(K_c)\). One nature idea for obtaining a $(\delta, \epsilon)$-stationary point of $J(K_c)$ is to perform
\begin{align} \label{warmup_alo}
K_c^{n+1} = K_c^n - \eta \nabla J_\delta (K_c^n)
\end{align}
provided that \(\nabla J_\delta(K_c)\) is accessible. However, \(\nabla J_\delta(K_c)\) is hard to compute in general. In addition, we focus on a model-free setting where we only have the access to the estimated  cost values. Building upon the insight in \eqref{warmup_alo}, we compute an estimate of the gradient \(\nabla J_\delta (K_c)\) using a zeroth-order oracle as outlined in Algorithm \ref{alg:DF_PO}.

The initialization of both Algorithm \ref{alo:NS} and Algorithm \ref{alg:DF_PO} involves establishing a feasible starting point \(K_c^0\). A typical initialization sets the \(D\) scale operator via choosing \(A_D\), \(B_D\), \(C_D\) matrices as zero, and making \(D_D\) the identity matrix. For the controller \(K\), we adopt initialization strategies such as PO-annealing methods, as suggested in \cite{ozaslan2022computing} and \cite{perdomo2021stabilizing}. Furthermore, to ensure the iterates $K_c^n$ and their perturbations remain within the feasible region \(\mathcal{K}_c\), one can choose small ($\delta^0$, $\alpha^0$) for Algorithm \ref{alo:NS} and small (\(\delta\), $\eta$) for Algorithm \ref{alg:DF_PO}.

 \begin{algorithm}[t!]
   \caption{Derivative-free optimization method with randomized smoothing}
   \label{alg:DF_PO}
\begin{algorithmic}
    \STATE {\bfseries Require:} feasible initial point $K_c^0$, stepsize $\eta > 0$, problem dimension $d \ge 1$, smoothing parameter $\delta$ and iteration number $N \ge 1$.
   \FOR{$n = 0, 1, \cdots, N-1$ }

   \STATE Sample $W^n\in \mathbb{R}^{n_{d}}$ uniformly at random over vectors such that $\| W\|_F = 1$.
   \STATE Compute $g^n = \frac{d}{2\delta}(J(K_c^n+\delta W^n)-J(K_c^n-\delta W^n))W^n$.
   \STATE Update $K_c^{n+1} = K_c^n - \eta g^n$.  
   \ENDFOR
   \STATE {\bfseries Output:} $K_c^R$ where $R\in \{ 0, 1,2,\cdots,N-1 \}$ is uniformly sampled.

\end{algorithmic}
\end{algorithm} 

\subsection{Theoretical Justifications}
As previously discussed in Section~\ref{sec:2}, the robust synthesis problem  \eqref{eq:opt_up} has an equivalent policy optimization formulation \eqref{eq:muopt}. 
In this context, \(K_c\) denotes the augmented controllers, integrating both the controller \(K\) and the \(D\) scale parameters.

If one can prove that the cost function in \eqref{eq:muopt} is coercive,  then the existing proof arguments in~\citet{guo2022global} and \citet{guo2023complexity} can be directly applied to provide theoretical justifications on convergence to stationary points. 
 However, as highlighted in \cite{Bompart2007}, the cost function \(J(K_c)\) might remain finite at the boundary of \(\mathcal{K}_c\), indicating situations where the system is not internally stable yet exhibits a finite cost \(J(K_c)\). 
The absence of coerciveness in the cost function poses a significant challenge for establishing the convergence behavior of model-free $\mu$-synthesis.
To address this, we draw inspiration from \cite{Bompart2007} and consider the closed-loop transfer function:

\begin{align}
    \mathcal{T}_{stab}(K_c,\mathsf{z}):= (\mathsf{z}I-A_{cl}(K_c))^{-1},
\end{align}
where \(A_{cl}(K_c)\) represents the closed-loop state matrix for \(\mathcal{F}_l(G_c,K_c)\). We assume that we have access to the following zeroth-order oracle that leads to a ``regularized" optimization problem:
\begin{align} \label{eq:opt_new}
    \min_{K_c \in \mathcal{K}_c } J_c(K_c) := \max\{ J(K_c), \lambda \| \mathcal{T}_{stab}(K_c,\mathsf{z}) \|_{\infty}  \},
\end{align}
where \(\lambda\) is a small positive parameter to be tuned. As commented in \cite{Bompart2007}, this adjustment addresses the system's internal stability concerns, while 
the modified cost \(J_c(K_c)\) is identical to  \(J(K_c)\) for any $K_c$ satisfying $J(K_c)\ge \lambda \| \mathcal{T}_{stab}(K_c,\mathsf{z}) \|_{\infty}$ (most $K_c$ satisfies this if $\lambda$ is sufficiently small).
Therefore, minimizing \eqref{eq:opt_new} with a sufficiently small \(\lambda\) effectively parallels the original policy optimization problem \eqref{eq:muopt}, specifically for interior points far away from the boundary of \( \mathcal{K}_c\). 
Now, we can establish the coerciveness of the \(J_c(K_c)\) in the following lemma\footnote{The coerciveness property does not come for free. The price is that the modification in the cost could potentially lead to new stationary points. For simplicity, this lemma fixes \(D_D\) as a constant matrix. We will relax this in the appendix.}

\begin{lemma}\label{lem1}
Suppose \(B := \left[ \begin{array}{cc|c} B_w & B_d & B_u 

\end{array} \right] \) and 
\(C := \left[ \begin{array}{cc|c} C_v^\top & C_e^\top & C_y^\top 

\end{array} \right]^\top \)  are full rank matrices. Then the objective function \(J_c(K_c)\) defined by \eqref{eq:opt_new} is coercive over the set \(\mathcal{K}_c\) in the sense that for any sequence \(\{K_c^l\}_{l=1}^\infty \subset \mathcal{K}_c\) we have
\begin{equation*}
    J_c(K_c^l) \rightarrow +\infty
\end{equation*}
if either \(\|K_c^l\|_F \rightarrow +\infty\), or \(K_c^l\) converges to an element in the boundary \(\partial \mathcal{K}_c\).
\end{lemma}
\begin{proof}
First noticing that it suffices to show that $\| \mathcal{T}_{stab}(K_c,\mathsf{z}) \|_{\infty}$ is coercive over the set $\mathcal{K}_c$. Suppose that there is a sequence $\{ K_c^l \}_{l=1}^\infty$ such that $K_c^l \to K_c^\dag \in \partial \mathcal{K}_c$. Clearly we have $\rho(A_{cl}(K_c^\dag))=1$, and there exists some $\omega_0$ such that the matrix $(e^{j\omega_0}I-A_{cl}(K_c^\dag) )$ becomes singular. Therefore, we have:
\begin{align*}
\| \mathcal{T}_{stab}(K_c^l, j \omega) \|_{\infty} &= \sup_{\omega\in[0, 2\pi]}\sigma_{\max}\left( (e^{j\omega}I-A_{cl}(K_c^l))^{-1} \right) \\
&\ge \sigma_{\max} \left( (e^{j\omega_0}I-A_{cl}(K_c^l) )^{-1} \right).
\end{align*}
Notice that  we have $\rho(A_{cl}(K_c^l)) < 1$ for each $l$ , and hence we have $\sigma_{\min} \left(e^{ j\omega_0}I-A_{cl}(K^l)\right)>0$, i.e. the smallest singular values of $(e^{ j\omega_0}I-A_{cl}(K_c^l))$ are positive for all $l$. By the continuity of $\sigma_{\min}(\cdot)$, we must have $\sigma_{\min} \left(e^{ j\omega_0}I-A_{cl}(K_c^l)\right) \to 0$ as $ K_c^l \to K_c^\dag \in \partial \mathcal{K}_c$. Hence we have $\sigma_{\max} \left(  (e^{ j\omega_0}I-A_{cl}(K_c^l))^{-1} \right) \to +\infty$ as $l \to \infty$. Then we have $\| \mathcal{T}_{stab}(K_c^l, z) \|_{\infty} \to +\infty$ as $K_c^l\to K_c^\dag\in \partial \mathcal{K}_c$. The proof of \(J_c(K_c^l) \rightarrow + \infty\) as \(\norm{K_c^l}_F \rightarrow + \infty\) will be given in the~appendix.

\end{proof}

Once the coerciveness is proved, we can slightly modify the existing proof techniques in~\citep{guo2022global,guo2023complexity} to show convergence guarantees to stationary points for model-free $\mu$-synthesis. We provide more discussions on this in the appendix. This part of extensions is actually quite straightforward, and hence omitted here.

\section{Numerical Experiments}
In this section, we present the numerical study to show the effectiveness of our proposed model-free methods across various examples.

\subsection{Doyle's Example}
We start with an illustrative example from \citep{doyle1985structured}, showing a scenario where the \(DK\)-iteration method has difficulties in converging to the optimal solution. We will show that Algorithm \ref{alg:DF_PO} successfully converges to the optimal solution for this example. Consider the following system:
\begin{align} \label{ex1}
    G = \begin{bmatrix} R & U \\ V & 0 \end{bmatrix}, \quad \text{where} \quad R = \begin{bmatrix} -1 & 1 \\ 0 & 1 \end{bmatrix},\,\, U = \begin{bmatrix} 0 \\ 1 \end{bmatrix},\,\, V = \begin{bmatrix} 1 & 0 \end{bmatrix}.
\end{align}
This system is coupled with a controller \(K=Q\in\mathbb{R}\) and an uncertainty set \(\mathbf{\Delta} = \{  \delta : \delta \in \mathbb{C} \}\). The upper bound of the \(\mu\)-synthesis is defined as the following problem (given the number of complex scalars being fewer than four, the upper bound is exactly the \(\mu\) value \citep{dullerud2013course}):
\begin{align} \label{ex1_up}
    \min_{Q \in \mathbb{R}} \min_{D \in \mathbf{D}} \sigma_{\text{max}} \left( D \mathcal{F}_l(G,K)D^{-1} \right)
\end{align}
where \(\mathcal{F}_l(G,K) = R+UQV\) and \(\mathbf{D} = \left\{ \bsmtx d & 0 \\ 0 & 1 \esmtx : d > 0 \right\}\). The optimization \eqref{ex1_up} can be rewritten as 
\begin{align} \label{ex1_up_last}
    \min_{Q \in \mathbb{R}, d > 0}  \sigma_{\text{max}} \left( \begin{bmatrix} -1 & d \\ Q/d & 1 \end{bmatrix} \right).
\end{align}
The optimal values for \(Q\) and \(d\) are \(Q^* = 0\) and \(d^* = 0\), yielding \(\mu^* = 1\). Employing the \(DK\)-iteration method, with \(Q\) fixed, the optimal \(d\) is \(d = \sqrt{Q}\), and with \(d\) fixed, the optimal \(Q\) is \(Q = d^2\). Thus iteratively, solving for either \(d\) or \(Q\) will immediately converge to \(Q=d^2\). Initializing with \(d^0 = 85\) and \(Q^0 = 72\), the iterates and contour lines generated by Algorithm \ref{alg:DF_PO} are drawn in Figure \ref{fig:Doyle}. The results clearly demonstrate the convergence of Algorithm \ref{alg:DF_PO} to the global minimum (\(Q^* = 0\) and \(d^* = 0\)), whereas the \(DK\)-iteration terminates at the suboptimal point \(Q = 8.485\) and \(d = 72\).
\begin{figure}[H]
    \centering
    \includegraphics[width=0.6\textwidth]{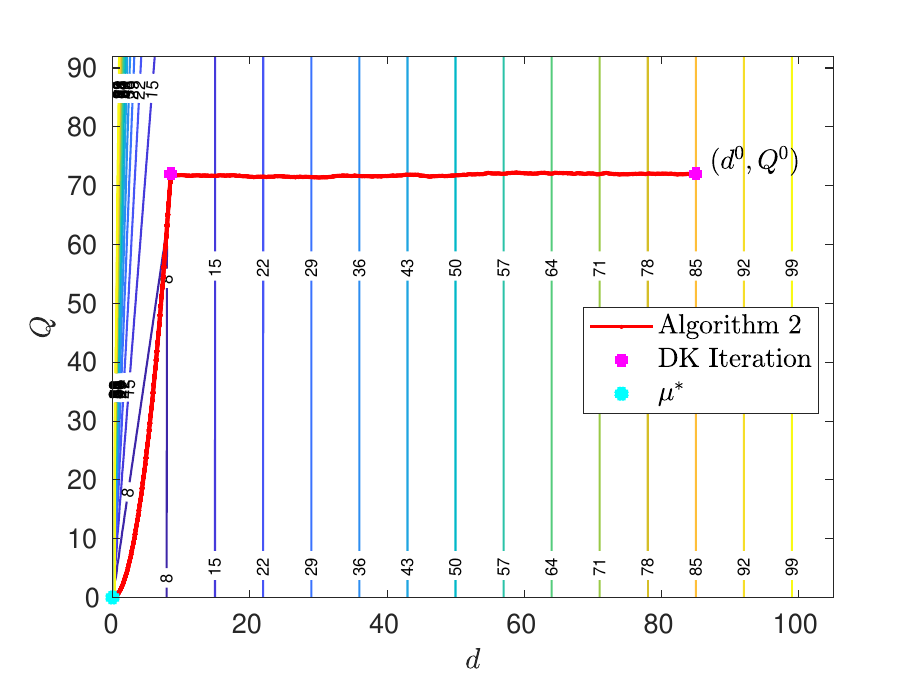}
    \caption{Algorithm \ref{alg:DF_PO} iterates in policy space for Doyle's example.}
    \label{fig:Doyle}
\end{figure}

\subsection{Higher Dimension Examples }
We now demonstrate the efficiency of our model-free methods on systems of higher dimensions. Specifically, Figure~\ref{fig:HigherDim} illustrates the relative error trajectories for systems with state dimension \(n_x = \{10, 20, 30\}\). The systems under consideration are structured according to Equation~\eqref{eq_system model}, with the matrices \((A, B_u, C_y)\) generated via the MATLAB function \texttt{drss}. The remaining matrices are realized by sampling from a standard normal distribution. Given that the \(A\) matrix is stable, the initial controller matrices \((A_K^0, B_K^0, C_K^0, D_K^0)\) are set to zero. Moreover, the \(D\)-scale operator matrices \((A_D^0, B_D^0, C_D^0)\) are initially set to zero, while the \(D_D^0\) matrix is initialized as the identity matrix. For all the experiments, \(D\)-scales of a state order of \(1\) is used. Table \ref{left_plot} presents a detailed results of our model-free methods compared to MATLAB's model-based \texttt{musyn} function. This shows that our model-free methods are not only comparable to, but in certain instances outperform \(DK\)-iteration.
\begin{table}[H]
  \caption{Comparison of our model-free methods with model-based method}
  \label{left_plot}
  \centering
  \begin{tabular}{lccc}
    \toprule
   $(n_x,n_w,n_u,n_d,n_v,n_e,n_y)$    & Algorithm 1 & Algorithm 2  & \texttt{musyn} \\ \midrule
    $(10,1,1,1,1,1,1)$   & 17.174 & 17.158 & 20.536    \\
    $(20,2,2,2,2,2,2)$   & 17.392 & 16.255& 18.681  \\
    $(30,3,3,3,3,3,3)$   & 38.417  & 37.051 & 32.501   \\
    \bottomrule
  \end{tabular}
\end{table}

\vspace{-0.1in}

\begin{figure}[h]
\minipage{0.5\textwidth}
   \centering
  \includegraphics[width=\linewidth]{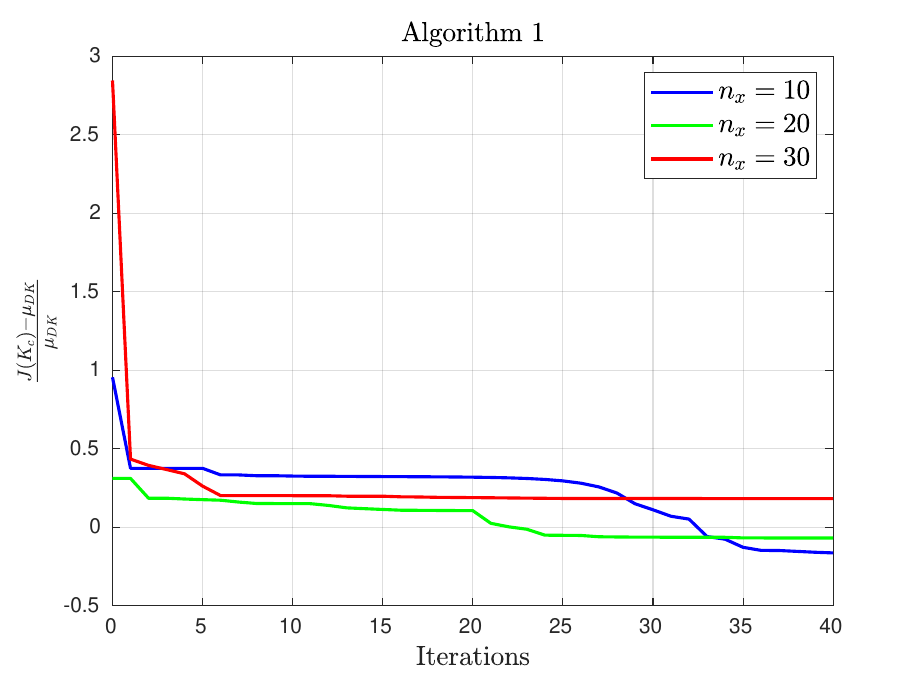}
\endminipage\hfill
\minipage{0.5\textwidth}%
   \centering
  \includegraphics[width=\linewidth]{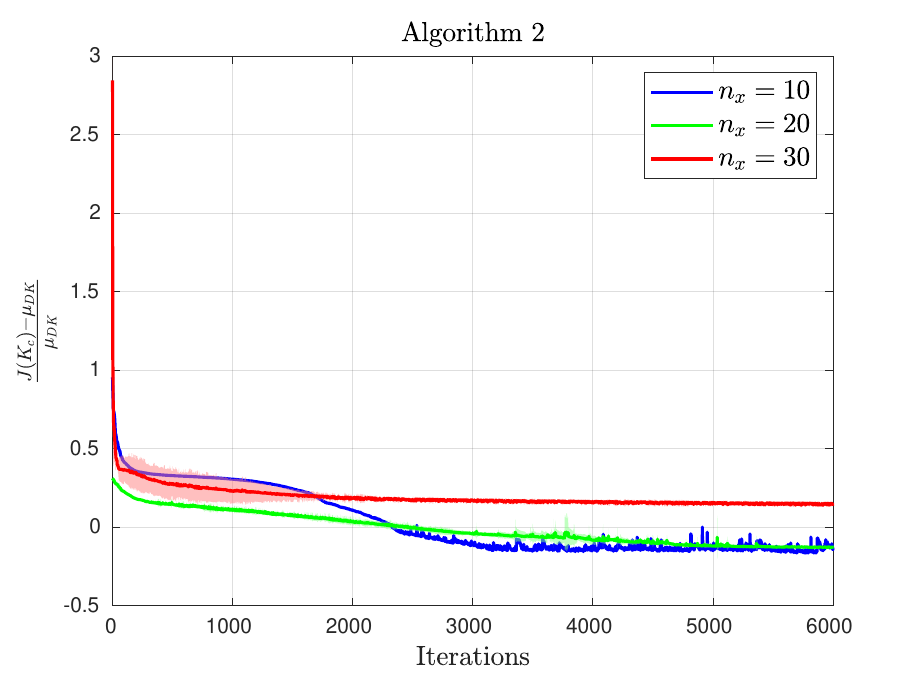}
\endminipage\hfill
\caption{Left: The plot illustrates the normalized deviation of trajectories from Algorithm~\ref{alo:NS} relative to the MATLAB \texttt{musyn} function outputs, denoted as \( \mu_{DK} \), across system states \( n_x = \{10, 20, 30\} \). Right: The plot illustrates the normalized deviation of trajectories from Algorithm~\ref{alg:DF_PO} relative to the MATLAB \texttt{musyn} function outputs, denoted as \( \mu_{DK} \), across system states \( n_x = \{10, 20, 30\} \). Solid lines depict the mean values, and the shaded regions represent the 98\% confidence intervals.}
\label{fig:HigherDim}
\end{figure}

\vspace{-0.4in}
\section{Conclusions}
This paper addresses model-free $\mu$-synthesis using nonsmooth optimization and direct policy search. 
We extend the nonsmooth optimization perspective in \citet{apkarian2011nonsmooth} to the general model-free output-feedback $\mu$-synthesis setting. 
Numerical study and theoretical justifications are both provided to demonstrate zeroth-order optimization as an efficient approach for model-free \(\mu\)-synthesis. 

\acks{The work of Darioush Kevian and Geir Dullerud is supported by NSF under Grant ECCS 1932735. The work of Xingang Guo and Bin Hu is generously supported by the NSF award CAREER-2048168. The work of Peter Seiler is supported by the U.S. Office of Naval Research (ONR) under Grant N00014-18-1-2209.}

\bibliography{main}

\begin{thebibliography}{55}
\providecommand{\natexlab}[1]{#1}
\providecommand{\url}[1]{\texttt{#1}}
\expandafter\ifx\csname urlstyle\endcsname\relax
  \providecommand{\doi}[1]{doi: #1}\else
  \providecommand{\doi}{doi: \begingroup \urlstyle{rm}\Url}\fi

\bibitem[Apkarian(2011)]{apkarian2011nonsmooth}
Pierre Apkarian.
\newblock Nonsmooth $\mu$-synthesis.
\newblock \emph{International Journal of Robust and Nonlinear Control}, 21\penalty0 (13):\penalty0 1493--1508, 2011.

\bibitem[Apkarian and Noll(2006{\natexlab{a}})]{apkarian2006controller}
Pierre Apkarian and Dominikus Noll.
\newblock Controller design via nonsmooth multidirectional search.
\newblock \emph{SIAM Journal on Control and Optimization}, 44\penalty0 (6):\penalty0 1923--1949, 2006{\natexlab{a}}.

\bibitem[Apkarian and Noll(2006{\natexlab{b}})]{apkarian2006nonsmooth}
Pierre Apkarian and Dominikus Noll.
\newblock Nonsmooth $\mathcal{H}_\infty$ synthesis.
\newblock \emph{IEEE Transactions on Automatic Control}, 51\penalty0 (1):\penalty0 71--86, 2006{\natexlab{b}}.

\bibitem[Arzelier et~al.(2011)Arzelier, Deaconu, Gumussoy, and Henrion]{arzelier2011h2}
Denis Arzelier, Georgia Deaconu, Suat Gumussoy, and Didier Henrion.
\newblock {$H_2$} for {HIFOO}.
\newblock In \emph{International Conference on Control and Optimization With Industrial Applications}, 2011.

\bibitem[Bompart et~al.(2007)Bompart, Noll, and Apkarian]{Bompart2007}
V.~Bompart, D.~Noll, and P.~Apkarian.
\newblock Second-order nonsmooth optimization for ${H}_\infty$ synthesis.
\newblock \emph{Numerische Mathematik}, 107:\penalty0 433--454, 2007.

\bibitem[Boyd et~al.(1994)Boyd, El~Ghaoui, Feron, and Balakrishnan]{boyd1994linear}
Stephen Boyd, Laurent El~Ghaoui, Eric Feron, and Venkataramanan Balakrishnan.
\newblock \emph{Linear Matrix Inequalities in System and Control Theory}, volume~15.
\newblock SIAM, 1994.

\bibitem[Bu et~al.(2019)Bu, Mesbahi, Fazel, and Mesbahi]{bu2019lqr}
Jingjing Bu, Afshin Mesbahi, Maryam Fazel, and Mehran Mesbahi.
\newblock {LQR} through the lens of first order methods: {D}iscrete-time case.
\newblock \emph{arXiv preprint arXiv:1907.08921}, 2019.

\bibitem[Burke et~al.(2005)Burke, Lewis, and Overton]{burke2005robust}
James~V Burke, Adrian~S Lewis, and Michael~L Overton.
\newblock A robust gradient sampling algorithm for nonsmooth, nonconvex optimization.
\newblock \emph{SIAM Journal on Optimization}, 15\penalty0 (3):\penalty0 751--779, 2005.

\bibitem[Burke et~al.(2020)Burke, Curtis, Lewis, Overton, and Sim{\~o}es]{burke2020gradient}
James~V Burke, Frank~E Curtis, Adrian~S Lewis, Michael~L Overton, and Lucas~EA Sim{\~o}es.
\newblock Gradient sampling methods for nonsmooth optimization.
\newblock \emph{Numerical Nonsmooth Optimization}, pages 201--225, 2020.

\bibitem[Curtis et~al.(2017)Curtis, Mitchell, and Overton]{curtis2017bfgs}
Frank~E Curtis, Tim Mitchell, and Michael~L Overton.
\newblock A {BFGS}-{SQP} method for nonsmooth, nonconvex, constrained optimization and its evaluation using relative minimization profiles.
\newblock \emph{Optimization Methods and Software}, 32\penalty0 (1):\penalty0 148--181, 2017.

\bibitem[Doyle(1985)]{doyle1985structured}
John~C Doyle.
\newblock Structured uncertainty in control system design.
\newblock In \emph{1985 24th IEEE Conference on Decision and Control}, pages 260--265. IEEE, 1985.

\bibitem[Duan et~al.(2021)Duan, Li, and Zhao]{duan2021optimization}
Jingliang Duan, Jie Li, and Lin Zhao.
\newblock Optimization landscape of gradient descent for discrete-time static output feedback.
\newblock \emph{arXiv preprint arXiv:2109.13132}, 2021.

\bibitem[Duchi et~al.(2012)Duchi, Bartlett, and Wainwright]{Duchi2012}
J.C. Duchi, P.L. Bartlett, and M.J. Wainwright.
\newblock Randomized smoothing for stochastic optimization.
\newblock \emph{SIAM Journal on Optimization}, 22\penalty0 (2):\penalty0 674--701, 2012.

\bibitem[Dullerud and Paganini(2013)]{dullerud2013course}
Geir~E Dullerud and Fernando Paganini.
\newblock \emph{A course in robust control theory: a convex approach}, volume~36.
\newblock Springer Science \& Business Media, 2013.

\bibitem[Fatkhullin and Polyak(2021)]{fatkhullin2020optimizing}
Ilyas Fatkhullin and Boris Polyak.
\newblock Optimizing static linear feedback: Gradient method.
\newblock \emph{SIAM Journal on Control and Optimization}, 59\penalty0 (5):\penalty0 3887--3911, 2021.

\bibitem[Fazel et~al.(2018)Fazel, Ge, Kakade, and Mesbahi]{pmlr-v80-fazel18a}
M.~Fazel, R.~Ge, S.~Kakade, and M.~Mesbahi.
\newblock Global convergence of policy gradient methods for the linear quadratic regulator.
\newblock In \emph{Proceedings of the 35th International Conference on Machine Learning}, volume~80, pages 1467--1476, 2018.

\bibitem[Furieri et~al.(2020)Furieri, Zheng, and Kamgarpour]{furieri2020learning}
Luca Furieri, Yang Zheng, and Maryam Kamgarpour.
\newblock Learning the globally optimal distributed {LQ} regulator.
\newblock In \emph{Learning for Dynamics and Control}, pages 287--297, 2020.

\bibitem[Gahinet and Apkarian(1994)]{Gahinet1994}
P.~Gahinet and P.~Apkarian.
\newblock A linear matrix inequality approach to ${H}_\infty$ control.
\newblock \emph{International Journal of Robust and Nonlinear Control}, 4:\penalty0 421--448, 1994.

\bibitem[Gahinet and Apkarian(2011)]{gahinet2011structured}
Pascal Gahinet and Pierre Apkarian.
\newblock Structured $\mathcal{H}_\infty$ synthesis in {MATLAB}.
\newblock \emph{IFAC Proceedings Volumes}, 44\penalty0 (1):\penalty0 1435--1440, 2011.

\bibitem[Ghadimi and Lan(2013)]{Ghadimi2013}
S.~Ghadimi and G.~Lan.
\newblock Stochastic first-and zeroth-order methods for nonconvex stochastic programming.
\newblock \emph{SIAM Journal on Optimization}, 23\penalty0 (4):\penalty0 2341--2368, 2013.

\bibitem[Goldstein(1977)]{goldstein1977optimization}
A.A. Goldstein.
\newblock Optimization of lipschitz continuous functions.
\newblock \emph{Mathematical Programming}, 13\penalty0 (1):\penalty0 14--22, 1977.

\bibitem[Gravell et~al.(2020)Gravell, Esfahani, and Summers]{gravell2020learning}
Benjamin Gravell, Peyman~Mohajerin Esfahani, and Tyler Summers.
\newblock Learning optimal controllers for linear systems with multiplicative noise via policy gradient.
\newblock \emph{IEEE Transactions on Automatic Control}, 66\penalty0 (11):\penalty0 5283--5298, 2020.

\bibitem[Gumussoy et~al.(2009)Gumussoy, Henrion, Millstone, and Overton]{gumussoy2009multiobjective}
Suat Gumussoy, Didier Henrion, Marc Millstone, and Michael~L Overton.
\newblock Multiobjective robust control with {HIFOO} 2.0.
\newblock \emph{IFAC Proceedings Volumes}, 42\penalty0 (6):\penalty0 144--149, 2009.

\bibitem[Guo and Hu(2022)]{guo2022global}
Xingang Guo and Bin Hu.
\newblock Global convergence of direct policy search for state-feedback $\mathcal{H}_{\infty}$ robust control: A revisit of nonsmooth synthesis with goldstein subdifferential.
\newblock In \emph{36th Conference on Neural Information Processing Systems, New Orleans, LA, Nov}, volume~28, 2022.

\bibitem[Guo et~al.(2023)Guo, Keivan, Dullerud, Seiler, and Hu]{guo2023complexity}
Xingang Guo, Darioush Keivan, Geir Dullerud, Peter Seiler, and Bin Hu.
\newblock Complexity of derivative-free policy optimization for structured $\mathcal{H}_{\infty}$ control.
\newblock In \emph{Thirty-seventh Conference on Neural Information Processing Systems}, 2023.

\bibitem[Hambly et~al.(2021)Hambly, Xu, and Yang]{hambly2021policy}
Ben Hambly, Renyuan Xu, and Huining Yang.
\newblock Policy gradient methods for the noisy linear quadratic regulator over a finite horizon.
\newblock \emph{SIAM Journal on Control and Optimization}, 59\penalty0 (5):\penalty0 3359--3391, 2021.

\bibitem[Honda and Seiler(2014)]{honda14ACC_HDD}
M.~Honda and P.~Seiler.
\newblock Uncertainty modeling for hard disk drives.
\newblock In \emph{American Control Conference}, pages 3341--3347, 2014.

\bibitem[Hu and Zheng(2022)]{hu2022connectivity}
Bin Hu and Yang Zheng.
\newblock Connectivity of the feasible and sublevel sets of dynamic output feedback control with robustness constraints.
\newblock \emph{IEEE Control Systems Letters}, 7:\penalty0 442--447, 2022.

\bibitem[Hu et~al.(2023)Hu, Zhang, Li, Mesbahi, Fazel, and Ba{\c{s}}ar]{hu2022towards}
Bin Hu, Kaiqing Zhang, Na~Li, Mehran Mesbahi, Maryam Fazel, and Tamer Ba{\c{s}}ar.
\newblock Toward a theoretical foundation of policy optimization for learning control policies.
\newblock \emph{Annual Review of Control, Robotics, and Autonomous Systems}, 6:\penalty0 123--158, 2023.

\bibitem[Jansch-Porto et~al.(2020{\natexlab{a}})Jansch-Porto, Hu, and Dullerud]{jansch2020policyMDP}
Joao~Paulo Jansch-Porto, Bin Hu, and Geir Dullerud.
\newblock Policy learning of {MDP}s with mixed continuous/discrete variables: A case study on model-free control of {M}arkovian jump systems.
\newblock In \emph{Learning for Dynamics and Control}, pages 947--957, 2020{\natexlab{a}}.

\bibitem[Jansch-Porto et~al.(2020{\natexlab{b}})Jansch-Porto, Hu, and Dullerud]{jansch2020convergence}
Joao~Paulo Jansch-Porto, Bin Hu, and Geir~E Dullerud.
\newblock Convergence guarantees of policy optimization methods for {M}arkovian jump linear systems.
\newblock In \emph{American Control Conference}, pages 2882--2887, 2020{\natexlab{b}}.

\bibitem[Jansch-Porto et~al.(2022)Jansch-Porto, Hu, and Dullerud]{jansch2022policy}
Joao~Paulo Jansch-Porto, Bin Hu, and Geir~E Dullerud.
\newblock Policy optimization for {M}arkovian jump linear quadratic control: {G}radient method and global convergence.
\newblock \emph{IEEE Transactions on Automatic Control}, 68\penalty0 (4):\penalty0 2475--2482, 2022.

\bibitem[Keivan et~al.(2021)Keivan, Havens, Seiler, Dullerud, and Hu]{keivan2021model}
Darioush Keivan, Aaron Havens, Peter Seiler, Geir Dullerud, and Bin Hu.
\newblock Model-free $\mu$ synthesis via adversarial reinforcement learning.
\newblock \emph{arXiv preprint arXiv:2111.15537}, 2021.

\bibitem[Kiwiel(2007)]{kiwiel2007convergence}
Krzysztof~C Kiwiel.
\newblock Convergence of the gradient sampling algorithm for nonsmooth nonconvex optimization.
\newblock \emph{SIAM Journal on Optimization}, 18\penalty0 (2):\penalty0 379--388, 2007.

\bibitem[Kiwiel(2010)]{kiwiel2010nonderivative}
Krzysztof~C Kiwiel.
\newblock A nonderivative version of the gradient sampling algorithm for nonsmooth nonconvex optimization.
\newblock \emph{SIAM Journal on Optimization}, 20\penalty0 (4):\penalty0 1983--1994, 2010.

\bibitem[Lin et~al.(2022)Lin, Zheng, and Jordan]{Lin2022}
T.~Lin, Z.~Zheng, and M.I. Jordan.
\newblock Gradient-free methods for deterministic and stochastic nonsmooth nonconvex optimization.
\newblock In \emph{Conference on Neural Information Processing Systems}, 2022.

\bibitem[Malik et~al.(2019)Malik, Pananjady, Bhatia, Khamaru, Bartlett, and Wainwright]{malik2019derivative}
Dhruv Malik, Ashwin Pananjady, Kush Bhatia, Koulik Khamaru, Peter Bartlett, and Martin Wainwright.
\newblock Derivative-free methods for policy optimization: {G}uarantees for linear quadratic systems.
\newblock In \emph{International Conference on Artificial Intelligence and Statistics}, pages 2916--2925, 2019.

\bibitem[Mohammadi et~al.(2021)Mohammadi, Zare, Soltanolkotabi, and Jovanovic]{mohammadi2021convergence}
Hesameddin Mohammadi, Armin Zare, Mahdi Soltanolkotabi, and Mihailo~R Jovanovic.
\newblock Convergence and sample complexity of gradient methods for the model-free linear quadratic regulator problem.
\newblock \emph{IEEE Transactions on Automatic Control}, 2021.

\bibitem[Ozaslan et~al.(2022)Ozaslan, Mohammadi, and Jovanovi{\'c}]{ozaslan2022computing}
Ibrahim~K Ozaslan, Hesameddin Mohammadi, and Mihailo~R Jovanovi{\'c}.
\newblock Computing stabilizing feedback gains via a model-free policy gradient method.
\newblock \emph{IEEE Control Systems Letters}, 7:\penalty0 407--412, 2022.

\bibitem[Packard et~al.(1993)Packard, Doyle, and Balas]{packard93musyn}
A.~Packard, J.~Doyle, and G.~Balas.
\newblock Linear, multivariable robust control with a $\mu$ perspective.
\newblock \emph{ASME Journal Dynamic Systems, Measurement, and Control}, 115\penalty0 (2B):\penalty0 426--438, 1993.

\bibitem[Perdomo et~al.(2021)Perdomo, Umenberger, and Simchowitz]{perdomo2021stabilizing}
Juan Perdomo, Jack Umenberger, and Max Simchowitz.
\newblock Stabilizing dynamical systems via policy gradient methods.
\newblock \emph{Advances in Neural Information Processing Systems}, 34, 2021.

\bibitem[Rathod et~al.(2021)Rathod, Bhadu, and De]{rathod2021global}
Santanu Rathod, Manoj Bhadu, and Abir De.
\newblock Global convergence using policy gradient methods for model-free {Markovian} jump linear quadratic control.
\newblock \emph{arXiv preprint arXiv:2111.15228}, 2021.

\bibitem[Scherer and Wieland(2004)]{Scherer2004}
C.~Scherer and S.~Wieland.
\newblock Linear matrix inequalities in control.
\newblock Lecture notes for a course of the dutch institute of systems and control, Delft University of Technology, 2004.

\bibitem[Scherer et~al.(1997)Scherer, Gahinet, and Chilali]{scherer1997multiobjective}
Carsten Scherer, Pascal Gahinet, and Mahmoud Chilali.
\newblock Multiobjective output-feedback control via lmi optimization.
\newblock \emph{IEEE Transactions on automatic control}, 42\penalty0 (7):\penalty0 896--911, 1997.

\bibitem[Tang and Zheng(2023)]{tang2023global}
Yujie Tang and Yang Zheng.
\newblock On the global optimality of direct policy search for nonsmooth $\mathcal{H}_{\infty}$ output-feedback control.
\newblock In \emph{2023 62nd IEEE Conference on Decision and Control (CDC)}, pages 6148--6153, 2023.

\bibitem[Wahlberg et~al.(2010)Wahlberg, Syberg, and Hjalmarsson]{Wahlberg2010}
B.~Wahlberg, M.B. Syberg, and H.~Hjalmarsson.
\newblock Non-parametric methods for $l_2$-gain estimation using iterative experiments.
\newblock \emph{Automatica}, 46\penalty0 (8):\penalty0 1376--1381, 2010.

\bibitem[Yang et~al.(2019)Yang, Chen, Hong, and Wang]{yang2019provably}
Zhuoran Yang, Yongxin Chen, Mingyi Hong, and Zhaoran Wang.
\newblock Provably global convergence of actor-critic: A case for linear quadratic regulator with ergodic cost.
\newblock \emph{Advances in neural information processing systems}, 32, 2019.

\bibitem[Zhang et~al.(2020)Zhang, Hu, and Ba{\c{s}}ar]{zhang2020stability}
Kaiqing Zhang, Bin Hu, and Tamer Ba{\c{s}}ar.
\newblock On the stability and convergence of robust adversarial reinforcement learning: A case study on linear quadratic systems.
\newblock \emph{Advances in Neural Information Processing Systems}, 33, 2020.

\bibitem[Zhang et~al.(2021{\natexlab{a}})Zhang, Hu, and Basar]{zhang2021policy}
Kaiqing Zhang, Bin Hu, and Tamer Basar.
\newblock Policy optimization for $\mathcal{H}_2$ linear control with $\mathcal{H}_\infty$ robustness guarantee: Implicit regularization and global convergence.
\newblock \emph{SIAM Journal on Control and Optimization}, 59\penalty0 (6):\penalty0 4081--4109, 2021{\natexlab{a}}.

\bibitem[Zhang et~al.(2021{\natexlab{b}})Zhang, Zhang, Hu, and Ba{\c{s}}ar]{zhang2021derivative}
Kaiqing Zhang, Xiangyuan Zhang, Bin Hu, and Tamer Ba{\c{s}}ar.
\newblock Derivative-free policy optimization for linear risk-sensitive and robust control design: Implicit regularization and sample complexity.
\newblock In \emph{Thirty-Fifth Conference on Neural Information Processing Systems}, 2021{\natexlab{b}}.

\bibitem[Zhao and You(2021)]{zhao2021primal}
Feiran Zhao and Keyou You.
\newblock Primal-dual learning for the model-free risk-constrained linear quadratic regulator.
\newblock In \emph{Learning for Dynamics and Control}, pages 702--714, 2021.

\bibitem[Zheng et~al.(2022)Zheng, Sun, Fazel, and Li]{zheng2022escaping}
Yang Zheng, Yue Sun, Maryam Fazel, and Na~Li.
\newblock Escaping high-order saddles in policy optimization for linear quadratic gaussian (lqg) control.
\newblock In \emph{2022 IEEE 61st Conference on Decision and Control (CDC)}, pages 5329--5334. IEEE, 2022.

\bibitem[Zheng et~al.(2023{\natexlab{a}})Zheng, Pai, and Tang]{zheng2023benign}
Yang Zheng, Chih-fan Pai, and Yujie Tang.
\newblock Benign nonconvex landscapes in optimal and robust control, part i: Global optimality.
\newblock \emph{arXiv preprint arXiv:2312.15332}, 2023{\natexlab{a}}.

\bibitem[Zheng et~al.(2023{\natexlab{b}})Zheng, Tang, and Li]{zheng2021analysis}
Yang Zheng, Yujie Tang, and Na~Li.
\newblock Analysis of the optimization landscape of linear quadratic gaussian ({LQG}) control.
\newblock \emph{Mathematical Programming}, 202:\penalty0 399--444, 2023{\natexlab{b}}.

\bibitem[Zhou et~al.(1996)Zhou, Doyle, and Glover]{zhou1996robust}
Kemin Zhou, John~Comstock Doyle, and Keith Glover.
\newblock \emph{Robust and Optimal Control}, volume~40.
\newblock Prentice Hall New Jersey, 1996.

\end{thebibliography}

\clearpage

\section*{Appendix}
In the main paper, we presented a partial proof of Lemma \ref{lem1}. In the appendix, we will complete the proof for this lemma, and provide some extra discussions on our theoretical justifications. 

\paragraph{Completing the proof for Lemma \ref{lem1}:}
To complete the proof,  we need to show that \( J_c(K_c^l) \rightarrow +\infty \) as \( \|K_c^l\|_{F} \rightarrow +\infty \). Suppose we have a sequence \(\{K_c^l\}\) satisfying \( \|K_c^l\|_F\rightarrow +\infty \). It is known that for a given \(K_c^l\), \( \|\mathcal{T}_{\text{stab}}(K_c^l, z) \|_{\infty}^2 \) is equivalent to the following supremum expression in the time domain:
\begin{align}
    \sup_{\mathbf{d}^l:\|\mathbf{d}^l\| \le 1} \sum_{t=0}^{\infty} \delta_t^\top \delta_t
\end{align}
where \(\delta_t\) is governed by an LTI system as below:
\begin{align}
\label{eq:supeq}
    \delta_{t+1} = A_{cl}(K_c^l) \delta_t + d_t^l, \quad \delta_0 = 0
\end{align}
Here, \(\mathbf{d}^l := \{d_0^l, d_1^l, \cdots \}\) represents the disturbance sequence, which can be chosen adversarially. Now, define:

\begin{align}
    \tilde{A} := \begin{bmatrix}
        A & 0 \\ 0 & 0
    \end{bmatrix}, \quad 
    \tilde{B} := \begin{bmatrix}
        0 & \hat{B} \\ I & 0
    \end{bmatrix}, \quad 
    \tilde{C} := \begin{bmatrix}
        0 & I \\ \hat{C} & 0
    \end{bmatrix}
\end{align}
where \(\hat{B}\) and \(\hat{C}\) are defined as

\begin{align*}
    \hat{B} := \underbrace{\left[\begin{array}{cc|c}
        B_w & B_d & B_u \\
    \end{array}\right]}_{B} 
    \begin{bmatrix} 
        0 & -I & 0 \\ I & 0 & 0 
    \end{bmatrix}, \quad 
    \hat{C} := \begin{bmatrix}
        0 & I \\ 0 & 0 \\ I & 0
    \end{bmatrix} 
    \underbrace{\left[ \begin{array}{c}
        C_{v} \\ C_e \\ \hline C_y
    \end{array}\right]}_{C}.
\end{align*}
Additionally, using the definition of the augmented controller \(K_{c}^l\) in \eqref{eq:Gmodel} as below:

\begin{align}
\begin{split}
    \xi_{t+1} &= \underbrace{\begin{bmatrix}
        A_K^l & 0 & 0 \\ 
        0 & A_D^l & 0 \\ 
        0 & 0 & A_D^l
    \end{bmatrix}}_{A_c^l} \xi_t 
    + \underbrace{\begin{bmatrix}
        B_K^l & 0 & 0 \\
        0 & \tilde{B}_D^l & 0 \\
        0 & 0 & \tilde{B}_D^l
    \end{bmatrix}}_{B_c^l} {y_c}_t, \\
    {u_c}_t &= \underbrace{\begin{bmatrix}
        C_K^l & 0 & 0 \\ 
        0 & \tilde{C}_D^l & 0 \\
        0 & 0 & \tilde{C}_D^l
    \end{bmatrix}}_{C_c^l} \xi_t 
    + \underbrace{\begin{bmatrix}
        D_K^l & 0 & 0 \\ 
        0 & \tilde{D}_D^l & 0 \\ 
        0 & 0 & \tilde{D}_D^l
    \end{bmatrix}}_{D_c^l} {y_c}_t,
\end{split}
\end{align}
and defining \(\tilde{K}_c^l := \left[ \begin{array}{c|c}

    A_c^l & B_c^l \\ \hline  
    C_c^l & D_c^l

\end{array}\right] \), we can show that \(A_{cl}(K_c^l) = \tilde{A} + \tilde{B} \tilde{K}_c^l \tilde{C}\). It can show that \(\tilde{B}\tilde{K}_c^l \tilde{C}\) is equal to 
\begin{equation}
\begin{split}
    \tilde{B}\tilde{K}_c^l \tilde{C} &= \begin{bmatrix} 0 & B \\ I & 0 \end{bmatrix} \left[ \begin{array}{cc}
        A_c^l & B_c^l \begin{bmatrix} 0 & I \\ 0 & 0 \\ I & 0 \end{bmatrix} \\
        \begin{bmatrix} 0 & -I & 0 \\ I & 0 & 0 \end{bmatrix} C_c^l & \begin{bmatrix} 0 & -I & 0 \\ I & 0 & 0 \end{bmatrix} D_c^l \begin{bmatrix} 0 & I \\ 0 & 0 \\ I & 0 \end{bmatrix}
    \end{array} \right] \begin{bmatrix} 0 & I \\ C & 0 \end{bmatrix} \\
    &= \underbrace{\begin{bmatrix} 0 & B \\ I & 0 \end{bmatrix}}_{\bar{B}} \underbrace{ \left[ \begin{array}{cc}
    A_c^l & \left[ \begin{array}{cc} B_K^l & 0 \\ 0 & 0 \\ \begin{bmatrix}
        B_D^l & 0
    \end{bmatrix} & 0
    \end{array} \right]
     \\ \left[ \begin{array}{ccc} 0 & \begin{bmatrix}
        -C_D^l \\ 0
    \end{bmatrix} & 0 \\ C_K^l & 0 & 0
    \end{array} \right] & \begin{bmatrix}
        0 & 0 \\ 0 & D_K^l
    \end{bmatrix}\end{array}\right]}_{\bar{K}_c^l} \underbrace{\begin{bmatrix} 0 & I \\ C & 0 \end{bmatrix}}_{\bar{C}}
\end{split}
\end{equation}
We can see that \(\bar{K}_c^l\) consists all the elements of \(K_c^l\) except \(D_D\) and since we assume \(D_D\) is a constant fix matrix, so if \(\norm{K_c^l}_F \rightarrow +\infty\), then \(\norm{\bar{K}_c^l}_F \rightarrow + \infty\). Let \(\mathbf{d}^l = \{d_0^l, 0, 0, \cdots \}\) with \(\norm{d_0^l} = 1\) such that \(\sigma_{\max}(\tilde{B} \tilde{K}_c^l \tilde{C}) = \norm{\tilde{B} \tilde{K}_c^l \tilde{C} d_0^l}\). Then, we have:
\begin{align*}
\| \mathcal{T}_{\text{stab}}(K_c^l, j \omega) \|_{\infty}^2 
&= \sup_{\mathbf{d}^l:\|\mathbf{d}^l\|\le 1} \sum_{t=0}^\infty \delta_t^T \delta_t \\
&\ge \sup_{\mathbf{d}^l:\|\mathbf{d}^l\|\le 1}  {\delta_1}^T \delta_1 + {\delta_2}^T \delta_2 \\ 
&=_{(1)} {d_0^l}^T d_0^l + (A_{\text{cl}} {d_0^l})^T (A_{\text{cl}} {d_0^l}) \\
&\ge {d_0^l}^T A_{\text{cl}}^T A_{\text{cl}} {d_0^l} \\
&= {d_0^l}^T (\tilde{A} + \tilde{B} \tilde{K}_c^l \tilde{C})^T (\tilde{A} + \tilde{B} \tilde{K}_c^l \tilde{C}) {d_0^l} \\
&= {d_0^l}^T (\tilde{A}^T \tilde{A} + \tilde{A}^T \tilde{B} \tilde{K}_c^l \tilde{C} + (\tilde{B} \tilde{K}_c^l \tilde{C})^T \tilde{A} + \tilde{C}^T \tilde{K}_c^{l^T} \tilde{B}^T \tilde{B} \tilde{K}_c^l \tilde{C} ){d_0^l} \\
&\ge {d_0^l}^T (\tilde{C}^T \tilde{K}_c^{l^T} \tilde{B}^T \tilde{B} \tilde{K}_c^l \tilde{C} ){d_0^l} \\
&= \|\tilde{B} \tilde{K}_c^l \tilde{C} d_0^l\|^2 \\
&\ge \sigma_{\min}(\bar{B}) \sigma_{\min}(\bar{C}) \sigma_{\max}(\bar{K}_c^l)
\end{align*}
where the inequality \((1)\) holds since we plugged in specific \(\mathbf{d}^l\) in \eqref{eq:supeq}. Since  \(B\) and \(C\) are full rank matrices, then we have \(\sigma_{\min}(\bar{B}) > 0\) and \(\sigma_{\min}(\bar{C}) > 0\). Since \(\norm{K_c^l}_F \rightarrow + \infty\) as \(l \rightarrow + \infty\), we have \(\sigma_{\max}(\bar{K}_c^l) \rightarrow +\infty\). Therefore, \(\norm{\mathcal{T}_{\text{stab}}(K_c^l,j\omega)}_{\infty}^2 \rightarrow + \infty\) as \(\norm{K_c^l}_F \rightarrow + \infty\). This completes the proof. $\blacksquare$

\paragraph{Relaxing assumptions.}
Also, it is worth mentioning that we can relax the assumptions of Lemma~\ref{lem1} and further remove the restictions on fixing $D_D$ via using the following regularized cost function:

\begin{align}
    \label{eq:new_cost}
\min_{K_c \in \mathcal{K}_c} J_c(K_c) :=\max\{J(K_c), \lambda_1 \| \mathbf{T}_{stab}(K_c,\mathsf{z})\|_{\infty}, \lambda_2 \| K_c \|_F \}
\end{align}
where \(\lambda_1\) and \(\lambda_2\) are small positive parameters to be tuned. We have the following corollary.
\begin{lemma}
    The objective function \(J_c(K_c)\) defined by \eqref{eq:new_cost} is coercive over the set \(\mathcal{K}_c\).
\end{lemma}
\begin{proof}
As stated in the proof of Lemma \ref{lem1}, as \(K_c^l \rightarrow K_c^\dag \in \partial \mathcal{K}_c\), we observe that \(\| \mathbf{T}_{\text{stab}}(K_c,\mathbf{{z}})\|_{\infty} \rightarrow +\infty \). Consequently, this implies \(J_c(K_c) \rightarrow +\infty\). Moreover, when \(\|K_c^l\|_F \rightarrow + \infty\), according to the definition in \eqref{eq:new_cost}, it follows that \(J_c(K_c) \rightarrow +\infty\). This completes the proof.
\end{proof}
For the above new cost function, one can change $D_D$ in a free way, and the resultant cost function is coercive without the extra rank assumptions in the original statement of Lemma \ref{lem1}.

\paragraph{More discussions on convergence to stationary points.}
As mentioned in the main paper, once the coerciveness is proved, we can slightly modify the proof techniques in~\citep{guo2022global,guo2023complexity} to show convergence guarantees to stationary points for model-free $\mu$-synthesis. 

We provide more explanations here. Consider the constrained optimization problem
$\min_{K_c \in \mathcal{K}_c} J_c (K_c)$ with $J_c$ being coercive. Then we can directly apply the arguments in \cite{guo2022global} to obtain the following facts.
\begin{enumerate}
    \item For any \(\gamma > J_c^*\), the sublevel set defined as \(\mathcal{S}_\gamma:=\{K_c\in \mathcal{K}_c: J_c(K_c)\le \gamma \}\) is compact. In addition,  there exists a strict separation between $\mathcal{S}_\gamma$ and the complement set of $\mathcal{K}_c$, i.e. \(\dist(\mathcal{S}_\gamma, \partial \mathcal{K}_c)>0\). 
    \item If \(K_c^0 \in \mathcal{K}_c\), then Algorithm \ref{alo:NS} with \(\mu_{\delta},\mu_{\epsilon} < 1\) converges to a stationary point with probability one\footnote{This can be proved via modifying the proofs  in \cite[Theorem C.2]{guo2022global} and \cite[Theorem 3.8]{kiwiel2010nonderivative}.}.
\end{enumerate}

Furthermore, by adapting the arguments outlined in  \cite[Theorem 3.7]{guo2023complexity}, we can construct a proof demonstrating the complexity of Algorithm \ref{alg:DF_PO} for finding $(\delta,\epsilon)$-stationary points. 
\end{document}